\documentclass[sn-mathphys,Numbered]{sn-jnl}
\usepackage{graphicx}%
\usepackage{multirow}%
\usepackage{amsmath,amssymb,amsfonts}%
\usepackage{amsthm}%
\usepackage{mathrsfs}%
\usepackage[title]{appendix}%
\usepackage{xcolor}%
\usepackage{textcomp}%
\usepackage{manyfoot}%
\usepackage{booktabs}%
\usepackage{algorithm}%
\usepackage{algorithmicx}%
\usepackage{algpseudocode}%
\usepackage{listings}%
\theoremstyle{thmstyleone}%
\newtheorem{theorem}{Theorem}%
\newtheorem{proposition}[theorem]{Proposition}%
\theoremstyle{thmstyletwo}%
\usepackage{bm}
\usepackage{float}
\usepackage{booktabs}
\usepackage{epsfig}
\usepackage{graphicx}
\usepackage{color}
\usepackage{ulem}
\usepackage{relsize}
\usepackage{exscale}
\usepackage{xcolor}
\usepackage{varioref}
\usepackage{comment}

\theoremstyle{thmstylethree}%
\newtheorem{ex}{Example}

\newtheorem{lem}{Lemma}
\newtheorem{cor}{Corollary}

\newcommand{\T}{\mathbb{P}}
\newcommand{\E}{\mathbb{E}}
\newcommand{\vphi}{\varphi}
\newcommand\dist{\buildrel d \over =}

\newcommand\al{\alpha}

\newcommand\la{\lambda}
\newcommand\ka{\kappa}
\def\leq{\leqslant}

\newcommand{\M}{\mathcal{M}}

\begin{document}

\title[Exact survival probability in E. Sparre Andersen's model]{On the exact survival probability by setting discrete random variables in E. Sparre Andersen's model}

\author*[1]{\fnm{Andrius} \sur{Grigutis}}\email{andrius.grigutis@mif.vu.lt}

\affil*[1]{\orgdiv{Institute of Mathematics}, \orgname{Vilnius university}, \orgaddress{\street{Naugarduko 24}, \city{Vilnius}, \postcode{LT-03225}, \country{Lithuania}}}

\abstract{In this work, we propose a simplification of the Pollaczek–Khinchine formula for the ultimate time survival (or ruin) probability calculation in exchange for a few assumptions on the random variables which generate the renewal risk model. More precisely, we show the expressibility of the distribution function 
$$
\mathbb{P}\left(\sup_{n\geqslant1}\sum_{i=1}^{n}(X_i-c\theta_i)<u\right),\,u\in\mathbb{N}_0
$$
via the roots of the probability generating function $G_{X-c\theta}(s)=1$, the expectation $\mathbb{E}(X-c\theta)$, and the probability mass function of $X-c\theta$. We assume that the random variables $X_1,\,X_2,\,\ldots$ and $c\theta_1,\,c\theta_2,\,\ldots$ are independent copies of $X$ and $c\theta$ respectively, $c>0$, $X$ and $c\theta$ are independent non-negative and integer-valued, and the support of $\theta$ is finite. We give few numerical outputs of the proven theoretical statements when the mentioned random variables admit some particular distributions.
}

\keywords{Ruin theory, renewal theory, random walk, survival probability, generating function, Pollaczek–Khinchine formula, initial values}


\pacs[MSC Classification]{60G50, 60J80, 91G05}

\maketitle

\section{Introduction}\label{sec:intr}

In applied probability, the following stochastic process
\begin{align}\label{eq:process}
W(0):=u,\,W(t):=u+ct-\sum_{i=1}^{N(t)}X_i,\,t>0,
\end{align}
is well known. Here $u\geqslant 0$, $c>0$, the non-negative random variables $X_1,\,X_2,\,\ldots$ are independent copies of $X$, and
$$
N(t)=\max\{n\in\mathbb{N}:\,\theta_1+\theta_2+\ldots+\theta_n\leqslant t\}
$$
is the renewal process generated by the non-negative random variables $\theta_1,\,\theta_2,\,\ldots$, which are the independent copies of $\theta$. 

The model \eqref{eq:process} is often met in queuing theory arguing that it represents the G/G/1 queue. The notation G/G/1 means that the queue length in a system with a single server is described by the interarrival times having an arbitrary distribution and the service times having some different distribution, see \cite{Asmussen}. In ruin theory and insurance mathematics, the process \eqref{eq:process} is known as E. Sparre Andersen's model or the renewal risk model. One may assume that $W(t)$ describes the insurer's wealth in time, where $u\geqslant 0$ denotes the initial surplus, $c>0$ represents the constant premium amount paid by the customers per unit of time and the subtracted random sum means payoffs caused by the random size claims at the random points in time, see for example \cite{thorin_1974} or \cite{Li} among voluminous other literature regarding the renewal risk models. Whatever is modeled by $W(t)$, one of the key aspects to know is whether $W(t)>0$ (or $W(t)\geqslant0$) for all such moments of time when the counting process $N(t)$ attains its jumps.

Let us define the {\it ultimate time survival probability} 
\begin{align}\label{eq:ult_time}\nonumber
\varphi(u)&:=\mathbb{P}\left(u+ct-\sum_{i=1}^{N(t)}X_i>0 \textit{ for all } t> 0\right)\\ \nonumber
&=\mathbb{P}\left(u+c\sum_{i=1}^{n}\theta_i-\sum_{i=1}^{N(\theta_1+\theta_2+\ldots+\theta_n)}X_i>0 \textit{ for all } n\geqslant 1\right)\\
&=\mathbb{P}\left(\sup_{n\geqslant 1}\sum_{i=1}^{n}\left(X_i-c\theta_i\right)<u\right),\,u\geqslant0
\end{align}
and, by the same argumentation, the {\it finite time survival probability}
\begin{align}\label{eq:fin_time}
\varphi(u,\,T):=\mathbb{P}\left(\sup_{1\leqslant n\leqslant T}\sum_{i=1}^{n}\left(X_i-c\theta_i\right)<u\right),\,u\geqslant0,\,T \in \mathbb{N}.
\end{align}
Clearly, both functions $\varphi(u)$ and $\varphi(u,\,T)$ in \eqref{eq:ult_time} and \eqref{eq:fin_time} respectively, serve the purpose to represent the distribution functions of the underlying random walk
\begin{align}\label{random_walk}
\left\{\sum_{i=1}^{n}\left(X_i-c\theta_i\right),\,n\in\mathbb{N}\right\}.
\end{align}
Perhaps the probability distribution of $\sup\limits_n\sum_{i=1}^{n}\left(X_i-c\theta_i\right)$ is one of the most natural numerical characteristics which is desired to know as this information allows us to compute any other characteristics of interest for the renewal risk models, see \cite{dickson_2023}, \cite{li_garrido_2005}, \cite{CHADJICONSTANTINIDIS2023127497}.

We further assume the existence of the first moments of independent and non-negative random variables $X$ and $\theta$, i.e. $\mathbb{E}X<\infty$ and $\mathbb{E}\theta<\infty$, where $X_1,\,X_2,\,\ldots$ and $\theta_1,\,\theta_2,\,\ldots$ in \eqref{random_walk} are independent copies of $X$ and $\theta$ respectively.

The computation of $\varphi(u,\,T)$ is far easier than $\varphi(u)$. Let us explain why. The finite time survival probability (see Proposition \ref{thm:fin_time} in Section \ref{sec:results} below) is based on the $T$-fold convolution of the distribution function $\mathbb{P}(X-c\theta<u),\,u\geqslant 0$. 
As $T\to\infty$, the general tool for $\varphi(u)$ calculation is the Pollaczek–Khinchine formula which states that, if $\mathbb{E}(c\theta-X)>0$, then 
\begin{align}\label{eq:P_C_f}
\varphi(u)&=e^{-A}\left(1+\sum_{n=1}^{\infty}\left(1-e^{-A}\right)^nH^{*n}(u)\right),\,u\geqslant0,
\end{align}
where
\begin{align*}
&A=\sum_{n=1}^{\infty}\frac{\mathbb{P}(S_n>0)}{n},\,
H(u)=\frac{F_{+}(u)}{F_{+}(\infty)},\,
F_{+}(u)=\mathbb{P}(S_{N^+}\leqslant u),\\
&N^+=\inf\{n\geqslant1:S_n>0\},\,
S_n=\sum_{i=1}^{n}\left(X_i-c\theta_i\right),
\end{align*}
and $H^{*n}$ denotes the $n$-fold convolution of the distribution function $H$, see \cite[eq. (10)]{EMBRECHTS}. Despite the formula \eqref{eq:P_C_f} admitting neat closed-form expression, its practical usage is complicated due to the convolution's $H^{*n}(u)$ calculation.

In this work, we demonstrate the closed-form expressiveness of the ultimate time survival probability $\vphi(u)$ via the roots of the probability-generating function of the random variable $X-c\theta$ where some other numerical characteristics of $X-c\theta$ are involved under certain assumptions of $X$ and $c\theta$. To formulate more precisely what is being done, let us first define the auxiliary notations and formulate assumptions.

Let $s\in\mathbb{C}$ be the complex number and define the probability-generating function for some arbitrary non-negative and integer-valued random variable $X$
\begin{align}\label{eq:gen_f}
G_{X}(s):=\mathbb{P}(X=0)+s\mathbb{P}(X=1)+s^2\mathbb{P}(X=2)+\ldots,\,|s|\leqslant 1.
\end{align}
The condition $|s|\leqslant1$ in \eqref{eq:gen_f} ensures the convergence of the provided power series in general. However, in many examples, $G_{X}(s)$ exists outside of the unit circle too. We now aim to define the probability-generating function of the random variable $X-c\theta$. This requires a couple of assumptions:\\
(a) The independent random variables $X$ and $c\theta$ are non-negative and integer-valued,\\
(b) The random variable $c\theta$ has the finite support, i.e. $\mathbb{P}(c\theta\leqslant m)=1$ for some $m\in\mathbb{N}$.

So, under assumptions (a) and (b), the probability generating function of the random variable $X-c\theta$, where $X$ and $\theta$ are independent, is
$$
G_{X-c\theta}(s)=G_{X}(s)G_{c\theta}\left(1/s\right), \, 0<|s|\leqslant1.
$$
Consequently, the expression of $\varphi(u)$ in this work is given under assumptions (a) and (b) too. In other words, we assume that the random walk \eqref{random_walk} runs over the integers $\geqslant -m$, $m\in\mathbb{N}$. The assumption (a) implies that it is sufficient to consider $\varphi(u)$ for $u\in\mathbb{N}_0$ only. 

Let us comment on that assumption (a) serves the purpose of the neater expressions of $\varphi(u)$. The assumption (a) may be considered as not too restrictive since every continuous or semi-continuous process can be discretized, i.e. ''we can always calculate the money in cents''. The assumption (b) excludes the opportunity ''to live happily ever after'' and the statement of (b) in this work is unavoidable for a number of reasons:
\begin{itemize}
\item we require the existence of $G_{X-c\theta}(s)$ in general,
\item we require $m-1$ roots of $G_{X-c\theta}(s)=1$ in $0<|s|\leqslant1$, $s\neq1$, see Lemma \ref{lem:roots_number} below,
\item we cannot get the information on all of $\varphi(u),\,u\in\mathbb{N}_0$ at once, see eq. \eqref{eq:ult_rec} below.
\end{itemize}

Again, assumption (b) might be considered as not too restrictive
from a practical perspective. The number $m$ in (b) can be arbitrarily large and, aiming to incorporate some inter-occurrence time $\theta$ of infinite support, we can define a new random variable $\theta_m$ whose distribution is
\begin{align}\label{syst:adjustment}
\begin{cases}
&\mathbb{P}\left(c\theta_m=k\right)=\mathbb{P}\left(c\theta=k\right),\,k\in\{0,\,1,\,\ldots,\,m-1\}\\
&\mathbb{P}\left(c\theta_m=m\right)=\mathbb{P}\left(c\theta\geqslant m\right)
\end{cases}.
\end{align}
Then, when $\theta^m_1,\,\theta^m_2,\,\ldots$ are independent copies of $\theta_m$,
$$
\mathbb{P}\left(\sup_{n\geqslant 1}\sum_{i=1}^{n}\left(X_i-c\theta^m_i\right)<u\right)
\leqslant
\mathbb{P}\left(\sup_{n\geqslant 1}\sum_{i=1}^{n}\left(X_i-c\theta_i\right)<u\right)
$$
for all $u\in\mathbb{N}_0$, and this can be considered as a more conservative approach to the modeled process since the ultimate time survival probability does not increase switching inter-occurrence time of infinite support to the finite one, see Example \ref{ex:puas} in Section \ref{sec:examples} below. Moreover, let's say that the model \eqref{eq:process} is generated by the non-negative integer-valued and independent random variables $X$ and $c\theta$, where the inter-occurrence time $\theta$ is of infinite support. Then, switching $c\theta$ to $c\theta_m$ of the finite support as provided in \eqref{syst:adjustment}, we can modify the probabilities of $X$, i.e. define a new random variable $X_m$, in a way that certain numerical characteristics of the random variables $X-c\theta$ and $X_m-c\theta_m$ are aligned. For instance, having some $X$, regardless of its support, we can take some of its value $l\neq 0$ and define the random variable $X_m$
\begin{align*}
\begin{cases}
&\mathbb{P}(X_m=0)=\mathbb{P}(X=0)+\frac{1}{l}\sum\limits_{i=1}^{\infty}i\mathbb{P}(c\theta=m+i)\\
&\mathbb{P}(X_m=k)=\mathbb{P}(X=k),\,k\in\mathbb{N},\,k\neq l\\
&\mathbb{P}(X_m=l)=\mathbb{P}(X=l)-\frac{1}{l}\sum\limits_{i=1}^{\infty}i\mathbb{P}(c\theta=m+i)
\end{cases},
\end{align*}
achieving $\mathbb{E}(X_m-c\theta_m)=\mathbb{E}(X-c\theta)$. In other words, we can balance between the ''softer punishment more often'' and the ''harsher punishment less often''. In addition, in view of $\varphi(u)\leqslant\varphi(u+1)$ for all $u\in\mathbb{N}_0$ and $\varphi(u)\to1$ as $u\to\infty$ under the net profit condition (see eq. \eqref{eq:lim_1} below), eq. \eqref{eq:ult_rec} implies the estimates
\begin{align}\label{ineq:est}
\inf_{i\geqslant m+1}\varphi(i)\mathbb{P}(X-c\theta\geqslant -m-1)\leqslant\vphi(0)-\sum_{i=1}^{m}\vphi(i)f(-i)\leqslant\mathbb{P}(X-c\theta\geqslant -m-1),
\end{align}
where $f(j)=\mathbb{P}(X-c\theta=j)$ for any integer $j$. This suggests that the error in estimating $\varphi(0)$ (the other survival probabilities too) when truncating the inter-occurrence time $\theta$ of infinite support is controlled by the probability tail mass of $\theta$.

Let us recall that the expectation $\mathbb{E}(X-c\theta)$ is called the {\it net profit condition} and it is said that the net profit condition holds if $G'_{X-c\theta}(1)=\mathbb{E}(X-c\theta)<0$, i.e. on average the claims' sizes are less than collected premiums. It is well known that $\varphi(u)=0$ for all $u\geqslant 0$ if $\mathbb{E}(X-c\theta)\geqslant0$ except some cases when $\mathbb{E}(X-c\theta)=0$ and $\mathbb{P}(X-c\theta=0)=1$, see \cite{AAJ}. 


As mentioned, the calculation of the survival (or ruin $\psi=1-\vphi$) probabilities is nothing but the convolution's calculation of the distribution function $\mathbb{P}(X-c\Theta<u)$, and of course, the finite convolution is far simpler than the ultimate one. Applying the total probability formula for the ultimate time survival probability under assumptions (a) and (b) we obtain the following recurrence relation (see Lemma \ref{lem:reccurence} and its proof below)
\begin{align}\label{eq:main_rec}
\varphi(u)=\sum_{i=1}^{u+m}\vphi(i)f(u-i),\,u\in\mathbb{N}_0,
\end{align}
where, as previuosly, $f(j)=\mathbb{P}(X-c\theta=j)$, $j=-m,\,-m+1,\,\ldots$ is the probability mass function. By setting $u=0,\,1,\,\ldots$ in formula \eqref{eq:main_rec} we get  
$$
\varphi(0)=\sum_{i=1}^{m}\vphi(i)f(-i),\,\varphi(1)=\sum_{i=1}^{m+1}\vphi(i)f(1-i),\,\ldots,
$$
which means that willing to get $\vphi(m),\,\vphi(m+1),\,\ldots$ we must know $\varphi(0),\,\varphi(1),\,\ldots,\,\varphi(m-1)$. Thus, the calculation of $\varphi(u),\,u\in\mathbb{N}_0$ turns into the finding of the initial values $\varphi(0),\,\varphi(1),\,\ldots,\,\varphi(m-1)$ for the recurrence \eqref{eq:main_rec}. 
Let
\begin{align}\label{eq:M}
\M=\left(\sup_{n\geqslant 1}\sum_{i=1}^{n}\left(X_i-c\theta_i\right)\right)^+,
\end{align}
where $a^+=\max\{0,\,a\}$, $a\in\mathbb{R}$ and $\pi_{i}=\mathbb{P}(\M=i),\,i\in\mathbb{N}_0$, be the probability mass function of the newly defined random variable $\M$. Then, according to \eqref{eq:ult_time} above,
\begin{align}\label{eq:phi_pi}
\vphi(u+1)=\mathbb{P}(\M\leqslant u)=\sum_{i=0}^{u}\pi_i,\,u\in\mathbb{N}_0
\end{align}
and the search of the initial values of $\vphi(0),\,\vphi(1),\,\ldots,\,\vphi(m-1)$ turns into the search of the initial values of the probability mass function of $\M$. However, the technique given in this work is applicable when finding the first initial values of $\pi_0,\,\pi_1,\,\ldots,\,\pi_{m-1}$ only. If $u=0$ and $m\to\infty$, the recurrence \eqref{eq:main_rec} implies
\begin{align}\label{eq:ult_rec}
\varphi(0)=\sum_{i=1}^{\infty}\vphi(i)f(-i)
\end{align}
and there is nothing better than the Pollaczek–Khinchine formula (eq. \eqref{eq:P_C_f} above) to compute the exact values of $\varphi(u)$ out of this ultimate recurrence \eqref{eq:ult_rec}. 

Let us explain how we obtain the desired initial values of $\pi_0,\,\pi_1,\,\ldots,\,\pi_{m-1}$. The defined random variable $\M$ (eq. \eqref{eq:M} above) admits the following distribution property
\begin{align}\label{eq:dist}
\M\dist(\M+X-c\theta)^+,
\end{align}
see \cite[p. 198]{Feller} or \cite[Lem. 4.1]{AG}. The net profit condition $\mathbb{E}(X-c\theta)<0$ ensures 
\begin{align}\label{eq:lim_1}
\lim_{u\to\infty}\vphi(u)=\mathbb{P}(\M<\infty)=1,
\end{align}
this can be proved by following the proof of \cite[Lem.1 ]{gervė2022distribution} or \cite[Lem. 4]{GJS}. Distributions' equality \eqref{eq:dist} implies the equality of the corresponding probability-generating functions
\begin{align}\label{eq:gen_eq}
\mathbb{E}s^{\M}=\mathbb{E}s^{(\M+X-c\theta)^+}=\mathbb{E}\left(\mathbb{E}s^{(\M+X-c\theta)^+}|\M\right),\,|s|\leqslant 1.
\end{align}

Replicating \eqref{eq:gen_eq} over the roots of
\begin{align}\label{eq:roots}
G_{X-c\theta}(s)=1,\,0<|s|\leqslant1,\,s\neq1
\end{align}
and letting $s\to1^{-}$ in derivative (with respect to $s$) of both sides of \eqref{eq:gen_eq}, we get a system of linear equations (see \eqref{eq:system} below) whose solution is the wanted probabilities $\pi_0,\,\pi_1,\,\ldots,\,\pi_{m-1}$, see Theorem \ref{thm:exact} below. Note that if $\mathbb{P}(c\theta\leqslant m)=1$ and $\mathbb{P}(X=0)>0$ there are exactly $m-1$ roots, counted with their multiplicities, of \eqref{eq:roots} in $0<|s|\leqslant1,\,s\neq1$, see Lemma \ref{lem:roots_number} below, and it is not clear in general what happens to the roots of \eqref{eq:roots} in the unit circle or the entire complex plane if $m\to\infty$.

On top of the described setup of the ultimate time survival probability $\varphi(u),\,u\in\mathbb{N}_0$ we can construct the ultimate time survival probability generating function $\Xi(s)$ too. Indeed,
$$
\Xi(s)=\sum_{i=0}^{\infty}\varphi(i+1)s^i=\sum_{i=0}^{\infty}\sum_{j=0}^{i}\pi_j s^i=\frac{G_{\M}(s)}{1-s},\,|s|<1,
$$
where the probability-generating function $G_{\M}(s)$ can be solved out of \eqref{eq:gen_eq}, see eq. \eqref{eq:gen} below.

The majority of the work in setting and solving the described system \eqref{eq:system} remains similar as in \cite{AG}, where the same discretized model \eqref{eq:process} with $\theta\equiv 1$ and $c\in\mathbb{N}$ is considered. So, the essence of this work can be described as following the reference \cite{AG} carefully and adjusting the corresponding formulas where needed.

Let us mention that the described initial values $\varphi(0),\,\varphi(1),\,\ldots,\,\varphi(m-1)$ can be found not using neither the generating function $G_{X-c\theta}(s)$ nor its roots, but rather calculating the approximate limits of certain recurrent sequences, see \cite{DAMARACKAS2014930}, \cite{GJ} and related works. However, the method of recurrent sequences is more extensive and fails either when $m\to\infty$.

\section{Main results}\label{sec:results}

In this section, we formulate the main statements for the finite time $\varphi(u,\,T)$ and the ultimate time $\varphi(u)$ survival probabilities calculation. The listed statements are proved in Section \ref{sec:proofs} below.

\bigskip

\begin{proposition}\label{thm:fin_time}
Let $\vphi(u,\,T)$ be the finite time survival probability of the stochastic process \eqref{eq:process}. Then, for all $T\in\mathbb{N}$,
\begin{align*}
\varphi(u,\,T)=\int\limits_{(-\infty,\,u)}\tilde{F}^{*(T-1)}(u-x)\,d\,\tilde{F}(x),\,u\geqslant 0,
\end{align*}
where $\tilde{F}^{*T}$ denotes the $T$-fold convolution of the distribution function $\tilde{F}(u)=\mathbb{P}(X-c\theta< u)$ and $\tilde{F}^{*0}(u):=1$.
\end{proposition}

\bigskip

Note that Proposition \ref{thm:fin_time} holds regardless of the assumptions (a) and (b) formulated in the previous section. Of course, the calculation of $\varphi(u,\,T)$ is more convenient assuming (a) and (b). This implies the following corollary of Proposition \ref{thm:fin_time}. 

\bigskip

\begin{cor}\label{cor:fin_t}
Assume that the stochastic process \eqref{eq:process} is generated by the non-negative, independent, and integer-valued random variables $X$ and $c\theta$, where $\mathbb{P}(c\theta\leqslant m)=1$ for some $m\in\mathbb{N}$. Then, for all $u\in\mathbb{N}_0$,
\begin{align*}
\vphi(u,\,1)=F(u-1),\,
\vphi(u,\,T)=\sum_{k=-m}^{u+1}\varphi(u-k,\,T-1)f(k),\,T\in\{2,\,3,\,\ldots\},
\end{align*}
where $f(j)=\mathbb{P}(X-c\theta=j)$ and $F(j)=\mathbb{P}(X-c\theta\leqslant j)$ for any integer $j$.
\end{cor}

\bigskip

Let us now turn to the ultimate time survival probability $\varphi(u)$. The next statement provides the relation of probabilities $\pi_0,\,\pi_1,\,\ldots,\,\pi_{m-1}$, probability-generating functions of $\M$ and $X-c\theta$, and the ultimate time survival probability generating function $\Xi(s)$. 

\bigskip

\begin{theorem}\label{thm:main_thm}
Assume that the stochastic process \eqref{eq:process} is generated by the non-negative, independent, and integer-valued random variables $X$ and $c\theta$, where $\mathbb{P}(c\theta\leqslant m)=1$ for some $m\in\mathbb{N}$. Then, under the net profit condition $c\mathbb{E}\theta-\mathbb{E}X>0$, the following equalities hold:
\begin{align}
&\sum_{i=0}^{m-1}\pi_i\sum_{j=i+1}^{m}
(1-s^{i-j})f(-j)=G_{\mathcal{M}}(s)(1-G_{X-c\theta}(s)),\,|s|\leqslant1,\label{eq:1}\\
&\sum_{i=0}^{m-1}\pi_i\sum_{j=i+1}^{m}(j-i)f(-j)=c\mathbb{E}\theta-\mathbb{E}X,\label{eq:2}\\
&\Xi(s)=\frac{1}{s^{m}(G_{X-c\theta}(s)-1)}\sum_{i=0}^{m-1}\pi_i\sum_{j=0}^{m-i-1}s^{j+i}F(-m+j),\,|s|<1\label{eq:gen},
\end{align}
where, as previously, $f(j)=\mathbb{P}(X-c\theta=j)$, $F(j)=\mathbb{P}(X-c\theta\leqslant j)$ for any integer $j$, and $s\in\mathbb{C}$ such that $s^m(1-G_{X-c\theta}(s))\neq 0$ in eq. \eqref{eq:gen}.
\end{theorem}

\bigskip

Let us mention that if $m=c=1$ and $\theta\equiv 1$, the equations \eqref{eq:2} and \eqref{eq:main_rec} imply 
$$
\pi_0=\vphi(1)=\frac{1-\mathbb{E}X}{\mathbb{P}(X=0)},\,\vphi(u)=\sum_{i=1}^{u+1}\vphi(i)\mathbb{P}(X=u+1-i),\,u\in\mathbb{N}_0,
$$
where $\mathbb{P}(X=0)>0$ because of the net profit condition. This is the well-known result for the homogeneous discrete-time risk model, see \cite{dickson_waters_1991}, \cite{Shiu}, \cite{shiu_1989}, and many other sources. 

Under assumptions (a), (b), and the net profit condition, eq. \eqref{eq:1} in Theorem \ref{thm:main_thm} implies the following corollary. 

\bigskip

\begin{cor}\label{cor:1}
Let $|\al|\leqslant 1$, $\al\neq\{0,\,1\}$ be the root of
$$
G_X(s)G_{c\theta}\left(\frac{1}{s}\right)=1
$$
and recall that $F(j)=\mathbb{P}(X-c\theta\leqslant j)$ for any integer $j$. Then
\begin{align}\label{eq:1_root}
\sum_{i=0}^{m-1}\pi_i\sum_{j=0}^{m-i-1}\al^{j+i}F(-m+j)=0.
\end{align}
Moreover, the $n$'th derivative
\begin{align}\label{eq:2_root}
\frac{d^n}{ds^n}\left(\sum_{i=0}^{m-1}\pi_i\sum_{j=0}^{m-i-1}s^{j+j}F(-m+j)\right)\Bigg|_{s=\al}=0
\end{align}
for all $n\in\{0,\,1,\,\ldots,\,r-1\}$, where $r\in\{1,\,2,\,\ldots,\,m-1\}$ denotes the multiplicity of the root $\alpha$.
\end{cor}

\bigskip

We recall that there are exactly $m-1$ roots of $G_X(s)G_{c\theta}\left(1/s\right)=1$ in $|s|\leqslant1$, $s\neq\{0,\,1\}$ counted with their multiplicities, see Lemma \ref{lem:roots_number} below.

\bigskip

Suppose all the roots $\al_1,\,\al_2,\,\ldots,\,\al_{m-1}\neq\{0,\,1\}$ of $
G_X(s)G_{c\theta}\left(1/s\right)=1
$ in $|s|\leqslant1$ are simple. Then, replicating the equation \eqref{eq:1_root} over these roots and including the equality \eqref{eq:2} we set up the following system
\begin{align}\nonumber
&\begin{pmatrix}
&\sum\limits_{j=0}^{m-1}\al_1^{j}F(-m+j)&\sum\limits_{j=0}^{m-2}\al_1^{j+1}F(-m+j)&\ldots&\al_1^{m-1}f(-m)\\
&\sum\limits_{j=0}^{m-1}\al_2^{j}F(-m+j)&\sum\limits_{j=0}^{m-2}\al_2^{j+1}F(-m+j)&\ldots&\al_2^{m-1}f(-m)\\
&\vdots&\vdots&\ddots&\vdots\\
&\sum\limits_{j=0}^{m-1}\al_{m-1}^{j}F(-m+j)&\sum\limits_{j=0}^{m-2}\al_{m-1}^{j+1}F(-m+j)&\ldots&\al_{m-1}^{m-1}f(-m)\\
&\sum_{j=1}^{m}jf(-j)&\sum_{j=2}^{m}(j-1)f(-j)&\ldots&f(-m)
\end{pmatrix}
\begin{pmatrix}
&\pi_0\\
&\pi_1\\
&\vdots\\
&\pi_{m-1}
\end{pmatrix}\\ \label{eq:system}
&\hspace{9cm}=
\begin{pmatrix}
&0\\
&0\\
&\vdots\\
&0\\
&c\E \theta -\E X
\end{pmatrix}
\end{align}
Let us denote the system \eqref{eq:system} as $A{\bm \pi}=B$. The matrix $A$ is Vandermonde-like and its determinant $|A|$ admits the following representation
\begin{align*}
|A|=(-1)^{m-1}f^{m}(-m)\prod_{j=1}^{m-1}(\al_{j}-1)
\prod_{1\leqslant i<j\leqslant m -1}\left(\al_j-\al_i\right),
\end{align*}
see Lemma \ref{lem:determ} below. Thus, $|A|\neq0$ as long as $\al_1,\,\al_2,\,\ldots,\,\al_{m-1}\neq\{0,\,1\}$ are the simple roots of $
G_X(s)G_{c\theta}\left(1/s\right)=1
$
in $|s|\leqslant1$ and $f(-m)>0$. In the next theorem, we provide 
the solution of \eqref{eq:system} and the exact values of $\varphi(u),\,u\in\mathbb{N}_0$.

\bigskip

\begin{theorem}\label{thm:exact}
Suppose $\al_1,\,\al_2,\,\ldots,\,\al_{m-1}\neq\{0,\,1\}$ are the simple roots of $
G_X(s)G_{c\theta}\left(1/s\right)=1
$
in $|s|\leqslant1$, and recall that $f(-m)=\mathbb{P}(X-c\theta=-m)>0$, $F(j)=\mathbb{P}(X-c\theta\leqslant j)$ for any integer $j$. Then
\begin{align*}
\tilde{\pi}_0&=\frac{1}{f(-m)}\prod_{j=1}^{m-1}\frac{\al_j}{\al_j-1},\\
\tilde{\pi}_1&=-\frac{
\sum\limits_{1\leq  j_1<\ldots<j_{m-2}\leq m-1}
\al_{j_1}\cdots\al_{j_{m-2}}}{f(-m)\prod\limits_{j=1}^{m-1}(\al_j-1)}
-\frac{F(-m+1)}{f(-m)}\tilde{\pi}_0,\\
\tilde{\pi}_2&=\frac{
\sum\limits_{1\leq  j_1<\ldots<j_{m-3}\leq m-1}
\al_{j_1}\cdots\al_{j_{m-3}}}{f(-m)\prod\limits_{j=1}^{m-1}(\al_j-1)}
-\frac{F(-m+2)}{f(-m)}\tilde{\pi}_0-\frac{F(-m+1)}{f(-m)}\tilde{\pi}_1,\\
&\,\,\,\vdots\\
\tilde{\pi}_{m-1}&=\frac{(-1)^{m+1}}{f(-m)}\prod\limits_{j=1}^{m-1}\frac{1}{\al_j-1}-\frac{1}{f(-m)}\sum_{i=0}^{m-2}\tilde{\pi}_i F(-1-i),
\end{align*}
where $\tilde{\pi}_i=\pi_i/(c\mathbb{E}\theta-\mathbb{E}X)$. Moreover, if $\tilde{\varphi}(u)=\varphi(u)/(c\mathbb{E}\theta-\mathbb{E}X)$, then
\begin{align*}
\tilde{\vphi}(0)&=(-1)^{m+1}\prod_{j=1}^{m-1}\frac{1}{\al_j-1},\,
\tilde{\vphi}(1)=\frac{1}{f(-m)}\prod_{j=1}^{m-1}\frac{\al_j}{\al_j-1},\\
\tilde{\vphi}(2)&=-\frac{F(-m+1)}{f(-m)}\tilde{\vphi}(1)\\
&\hspace{0.5cm} +\prod_{j=1}^{m-1}\frac{1/f(-m)}{\al_j-1}\left(\prod_{j=1}^{m-1}\al_j-\sum\limits_{1\leq  j_1<\ldots<j_{m-2}\leq m-1}
\al_{j_1}\cdots\al_{j_{m-2}}\right),\\
\tilde{\varphi}(3)&=-\frac{F(-m+1)}{f(-m)}\tilde{\vphi}(2)-\frac{F(-m+2)}{f(-m)}\tilde{\vphi}(1)+\prod_{j=1}^{m-1}\frac{1/f(-m)}{\al_j-1}\\
&\hspace{-1cm}\times\left(\prod_{j=1}^{m-1}\al_j-\sum_{1\leq  j_1<\ldots<j_{m-2}\leq m-1}
\al_{j_1}\cdots\al_{j_{m-2}}
+\sum_{1\leq  j_1<\ldots<j_{m-3}\leq m-1}
\al_{j_1}\cdots\al_{j_{m-3}}\right),\\
&\,\,\vdots\\
\tilde{\vphi}(m)&=-\frac{1}{f(-m)}\sum_{i=1}^{m-1}F(-i)\tilde{\vphi}(i)
+\prod_{j=1}^{m-1}\frac{1/f(-m)}{\al_j-1}
\times\Bigg(\prod_{j=1}^{m-1}\al_j-\\
&\hspace{-1cm}\sum_{1\leq  j_1<\ldots<j_{m-2}\leq m-1}
\al_{j_1}\cdots\al_{j_{m-2}}
+\sum_{1\leq  j_1<\ldots<j_{m-3}\leq m-1}
\al_{j_1}\cdots\al_{j_{m-3}}+\ldots+(-1)^{m+1}\Bigg)
\end{align*}
and
\begin{align}\label{eq:reccurence}
\vphi(u)=\frac{1}{f(-m)}\left(\vphi(u-m)-\sum_{i=1}^{u-1}\vphi(i)f(u-m-i)\right),\,u=m,\,m+1,\,\ldots
\end{align}
\end{theorem}

If $f(-m)>0$ and there are multiple roots among $\al_1,\,\al_2,\,\ldots,\,\al_{m-1}\neq\{0,\,1\}$, in order to avoid identical lines, we then shall modify the system \eqref{eq:system} by replacing its line (or lines)
by derivatives, as provided by the equality \eqref{eq:2_root}, see Example \ref{ex:double_root} below. Such a modified system remains non-singular too because, roughly, the derivative represents the linear mapping, see \cite[Lem 4.3]{AG} for more detailed proof.

The system \eqref{eq:system} may require modifications not only due to the multiple roots. If $\mathbb{P}(X-c\theta>i)=1$ for some $i>-m$ under assumption $\mathbb{P}(c\theta=m)>0$, i.e. the values of the random variable $X$ (with positive probability) start by some value which is greater than zero, we then shall rebuild the system \eqref{eq:system} according to the equations \eqref{eq:1_root}, including \eqref{eq:2_root} in case there are multiple roots, and \eqref{eq:2}. In other words, if $\mathbb{P}(c\theta=m)>0$ and $\mathbb{P}(X\geqslant i)=1$ for some $i\in\mathbb{N}$, the modification is necessary in order to avoid the zero-columns in system's matrix in \eqref{eq:system}. Let us observe that (if $\mathbb{P}(c\theta=m)>0$) $\mathbb{P}(X\geqslant m)=1$ violates the net profit condition.
\section{Lemmas}

In this section, we formulate and prove several auxiliary statements that Section \ref{sec:results} is based on.

\bigskip

\begin{lem}\label{lem:reccurence}
The ultimate time survival probability \eqref{eq:ult_time} satisfies the following recurrence relation
$$
\varphi(u)=\sum_{i=1}^{u+m}\vphi(i)f(u-i),\,u\in\mathbb{N}_0.
$$
\end{lem}
\begin{proof}
Since the random variables $X_1-c\theta_1,\,X_2-c\theta_2,\,\ldots$ are independent and identically distributed, and $\mathbb{P}(X-c\theta\geqslant -m)=1$, then
\begin{align*}
\vphi(u)&=\mathbb{P}\left(\bigcap_{n=1}^{\infty}\left\{\sum_{i=1}^{n}(X_i-c\theta_i)<u\right\}\right)\\
&=\mathbb{P}\left(\bigcap_{n=2}^{\infty}\left\{X_1-c\theta_1+\sum_{i=2}^{n}(X_i-c\theta_i)<u\right\},\,X_1-c\theta_1<u\right)\\
&=\sum_{i=-m}^{u-1}\vphi(u-i)f(i)=\sum_{i=1}^{u+m}\vphi(i)f(u-i).
\end{align*}
\end{proof}

\begin{lem}\label{lem:determ}
Let $A$ be the system's matrix from \eqref{eq:system} and denote $f^m(-m)=\left(\mathbb{P}(X-c\theta=-m)\right)^m$. Then, the determinant of $A$ is
$$
|A|=(-1)^{m-1}f^{m}(-m)\prod_{j=1}^{m-1}(\al_{j}-1)
\prod_{1\leqslant i<j\leqslant m -1}\left(\al_j-\al_i\right).
$$
Moreover, the minors of the last line of $A$ are:
\begin{align*}
M_{m,\,1}&=f^{m-1}(-m)\prod_{i=1}^{m-1}\al_i\prod_{1\leqslant i<j\leqslant m-1}(\al_j-\al_i),\\
M_{m,\,2}&=f(-m)^{m-1}\prod_{1\leqslant i<j\leqslant m-1}(\al_j-\al_i)
\sum_{1\leqslant  j_1<\ldots<j_{m-2}\leqslant m-1}
\al_{j_1}\cdots\al_{j_{m-2}}\\
&+\frac{F(-m+1)}{f(-m)}M_{\ka,\,1},\\
M_{m,\,3}&=f(-m)^{m-1}\prod_{1\leq i<j\leq m-1}(\al_j-\al_i)
\sum_{1\leq  j_1<\ldots<j_{m-3}\leq m-1}
\al_{j_1}\cdots\al_{j_{m-3}},\\
&-\frac{F(-m+2)}{f(-m)}M_{\ka,\,1}
+\frac{F(-m+1)}{f(-m)}M_{\ka,\,2},\\
&\vdots\\
M_{m,\,m}&=f(-m)^{m-1}\prod_{1\leq i<j\leq m-1}(\al_j-\al_i)
+(-1)^{m}\frac{F(-1)}{f(-m)}M_{m,\,1}\\
&+(-1)^{m+1}
\frac{F(-2)}{f(-m)}M_{m,\,2}
+\ldots
+(-1)^{2m-1}\frac{F(-m+2)}{f(-m)}M_{\ka,\,\ka-2}\\
&+\frac{F(-m+1)}{f(-m)}M_{m,\,m-1}.
\end{align*}
\end{lem}

\begin{proof}
The proof is pretty technical and is based on elementary calculations of the determinants, see the reference \cite[Lem 4.2 and pp. 5192-5193]{AG} where the same determinants with $\theta\equiv1$ and $c\in\mathbb{N}$ are calculated more extensively.
\end{proof}

\begin{lem}\label{lem:roots_number}
Let $s\in\mathbb{C}$. For the non-negative, independent and integer-valued random variables $X$ and $c\theta$ assume $\mathbb{P}(c\theta\leqslant m)=1$ and $\mathbb{P}(X-c\theta=-m)>0$, i.e. $\mathbb{P}(X=0)>0$ and $\mathbb{P}(c\theta=m)>0$. Then, there are exactly $m-1$ roots, counted with their multiplicities, of
\begin{align}\label{eq:gen_1_roots}
G_{X-c\theta}(s)=1
\end{align}
in $|s|\leqslant 1$, $s\neq\{0,\,1\}$.

\end{lem}
\begin{proof}
We shall follow the reference \cite[Sec. 4]{GJ}. Let us temporarily assume $\mathbb{P}(c\theta=m)=1$. Then, the equation \eqref{eq:gen_1_roots} becomes
$$
G_X(s)=s^m.
$$
Since $\mathbb{P}(X=0)>0$, the latter equation has exactly $m-1$ roots in $0<|s|\leqslant1$, $s\neq1$, counted with their multiplicities. This fact is implied by (i) the fundamental theorem of algebra, which states that (without loosing the generality) the equation $s^m=1$ has $m$ roots counted with their multiplicities (or $m-1$ not counting $s=1$), (ii) the estimate 
$$
|G_X(s)|\leqslant 1 < |\la s^m|,\,\la>1,
$$
and (iii) Rouché's theorem, see \cite{Conway}.
Then, declining the requirement $\mathbb{P}(c\theta=m)=1$, we can rewrite the equation \eqref{eq:gen_1_roots} as follows
\begin{align*}
\left(\mathbb{P}(X=0)+\mathbb{P}(X=0)s+\ldots\right)\left(\mathbb{P}(c\theta=0)+\frac{\mathbb{P}(c\theta=1)}{s}+\ldots+\frac{\mathbb{P}(c\theta=m)}{s^m}\right)=1.
\end{align*}
Multiplying both sides of the last equation by $s^m$, we obtain
\begin{align*}
\left(\mathbb{P}(X=0)+\mathbb{P}(X=0)s+\ldots\right)\left(\mathbb{P}(c\theta=0)s^m+\mathbb{P}(c\theta=1)s^{m-1}+\ldots+\mathbb{P}(c\theta=m)\right)=s^m,
\end{align*}
where the left-hand side of the last equation is the probability-generating function of the non-negative discrete and integer-valued distribution and has the same number of roots in $|s|\leqslant1$, $s\neq\{0,\,1\}$ as in the previous case when $\mathbb{P}(c\theta=m)=1$ was assumed.  
\end{proof}

\section{Proofs of the main results}\label{sec:proofs}

In this section, we prove all of the statements formulated in Section \ref{sec:results}.

\begin{proof}[Proof of Proposition \ref{thm:fin_time}]
The proof is pretty evident. If $T=1$, then
$$
\vphi(u,\,1)=\int\limits_{(-\infty,\,u)}\,d\,\tilde{F}(x)=\mathbb{P}(X-c\theta<u).
$$
If $T=2$, then
\begin{align*}
\vphi(u,\,2)&=\mathbb{P}\left(\{X_1-c\theta_1<u\}\cap\{X_1-c\theta_1+X_2-c\theta_2<u\}\right)\\
&=\int\limits_{(-\infty,\,u)}\mathbb{P}(X-c\theta<u-x)\,d\,\mathbb{P}(X-c\theta<x)
=\int\limits_{(-\infty,\,u)}\tilde{F}^{*1}(u-x)\,d\,\tilde{F}(x)
\end{align*}
and, by mathematical induction,
\begin{align*}
\vphi(u,\,T)=\int\limits_{(-\infty,\,u)}\tilde{F}^{*(T-1)}(u-x)\,d\,\tilde{F}(x).
\end{align*}
\end{proof}

\begin{proof}[Proof of Corollary \ref{cor:fin_t}]
The expression of $\varphi(u,\,1)$ is evident. The formula of $\varphi(u,\,T)$ for $T\in\{2,\,3,\,\ldots\}$ is implied as follows
\begin{align*}
\varphi(u,\,T)&=\mathbb{P}\left(\bigcap\limits_{n=1}^{T}\left\{\sum_{i=1}^{n}(X_i-c\theta_i)<u\right\}\right)\\
&=\mathbb{P}\left(\bigcap\limits_{n=2}^{T}\left\{\sum_{i=1}^{n}(X_i-c\theta_i)<u\right\}\cap\{X_1-c\theta_1<u\}\right)\\
&=\sum_{k=-m}^{u-1}\mathbb{P}\left(\bigcap\limits_{n=1}^{T-1}\left\{\sum_{i=1}^{n}(X_i-c\theta_i)<u-(X_1-c\theta_1)\right\}\right)\mathbb{P}(X-c\theta=k)\\
&=\sum_{k=-m}^{u-1}\varphi(u-k,\,T-1)f(k).
\end{align*}
\end{proof}

\begin{proof}[Proof of Theorem \ref{thm:main_thm}]
The equality $\M\dist\left(\M+X-c\theta\right)^+$ (see eq. \eqref{eq:dist} above) implies the equality of the probability-generating functions of the underlying distributions. Then, applying the rule of total expectation when having in mind that $\T (X-c\theta\geqslant-m)=1$ and $\mathbb{P}(\M<\infty)=1$ due to the net profit condition (see eq. \eqref{eq:lim_1} above), we obtain
\begin{align*}
\E s^\M&=\E s^{\left(\M+X-c\theta\right)^+}=\E\left(\E \left(s^{\left(\M+X-c\theta\right)^+}|\M\right)\right)\\
&=\sum_{i=0}^{m-1}\pi_i\E s^{(i+X-c\theta)^+} + \sum_{i=m}^{\infty}\pi_i \E s^{i+X-c\theta}\\
&=\sum_{i=0}^{m-1}\pi_i\left(\sum_{j=-m}^{-i}f(j)+\sum_{j=-i+1}^{\infty}s^{i+j}f(j)\right)
+\E s^{X-c\theta}  \sum_{i=m}^{\infty}\pi_i s^i\\
&=\sum_{i=0}^{m-1}\pi_i\left(\sum_{j=-m}^{-i}f(j)+\sum_{j=-i+1}^{\infty}s^{i+j}f(j)\right)
+G_{X-c\theta}(s)\left(G_{\M}(s)-\sum_{i=0}^{m-1}\pi_i s^i\right)\\
&=\sum_{i=0}^{m-1}\pi_i\left(\sum_{j=-m}^{-i}f(j)+s^i\sum_{j=-i+1}^{\infty}s^{j}f(j)-s^i\sum_{j=-m}^{\infty}s^jf(j)\right)
+G_{X-c\theta}(s)G_{\M}(s)\\
&=\sum_{i=0}^{m-1}\pi_i\sum_{j=-m}^{-i-1}f(j)(1-s^{i+j})+
G_{\M}(s)G_{X}(s)G_{c\theta}(1/s),
\end{align*}
which implies
\begin{align}\label{eq:1_in_proof}
\sum_{i=0}^{m-1}\pi_i\sum_{j=i+1}^{m}f(-j)(1-s^{i-j})
=G_{\M}(s)(1-G_X(s)G_{c\theta}(1/s))
\end{align}
and the equality \eqref{eq:1} is proved. The second Theorem's equality \eqref{eq:2} is implied by the proven equality \eqref{eq:1_in_proof} by taking the derivative of both sides with respect to $s$ and letting $s\to1^{-}$, see \cite[Lem. 9]{GJS} and \cite[p. 5191]{AG}. 
\end{proof}

\bigskip

\begin{proof}[Proof of Corollary \ref{cor:1}]
If $|\al|\leqslant 1$, $\al\neq\{0,\,1\}$ is the root of
\begin{align*}
G_X(s)G_{c\theta}\left(\frac{1}{s}\right)=1,
\end{align*}
the equality \eqref{eq:1} implies
\begin{align}\label{eq:in_cor}
\sum_{i=0}^{m-1}\pi_i\sum_{j=i+1}^{m}f(-j)(1-s^{i-j})=0.
\end{align}
Since 
$$
(1-s^{i-j})=\left(1+s^{-1}+\ldots+s^{i-j+1}\right)(1-s^{-1}),
$$
we multiply the both sides of \eqref{eq:in_cor} by $1-s^{-1}$, $s\neq0$ and obtain
\begin{align*}
&(1-s^{-1})\sum_{i=0}^{m-1}\pi_i\sum_{j=i+1}^{m}f(-j)(1-s^{i-j})=\sum_{i=0}^{m-1}\pi_i\sum_{j=i+1}^{m}f(-j)\sum_{l=0}^{-i+j-1}s^{-j}\\
&=\sum_{i=0}^{m-1}\pi_i\sum_{l=0}^{m-i-1}s^{-l}\sum_{j=i+1+l}^{m}f(-j)
=\sum_{i=0}^{m-1}\pi_i\sum_{l=0}^{m-i-1}s^{-l}F(-i-1-l)\\
&=s^{-m+1}\sum_{i=0}^{m-1}\pi_i\sum_{j=0}^{m-1-i}s^{j+i}F(-m+j)=0.
\end{align*}
\end{proof}

\begin{proof}[Proof of Theorem \ref{thm:exact}]
As the roots $\al_1,\,\al_2,\,\ldots,\,\al_{m-1}$ are assumed to be simple and $f(-m)>0$, the system \eqref{eq:system} has the solution $(\pi_0,\,\pi_1,\,\ldots,\,\pi_{m-1})$ where
$$
\pi_{i}=\frac{(-1)^{m+i+1}M_{m,\,i+1}}{|A|}(\E (c\theta-X)),\,i=0,\,1,\,\ldots,\,m-1
$$
and the minors $M_{m,\,1},\,M_{m,\,2},\,\ldots,\,M_{m,\,m}$, including determinant $|A|$, are given in Lemma \ref{lem:determ} above. Then
$$
\varphi(i+1)=\sum_{j=0}^{i}\pi_i,\,i=0,\,1,\,\ldots,\,m-1,
$$
due to eq. \eqref{eq:phi_pi}. The formula of $\varphi(0)$ is implied by the recurrence \eqref{eq:main_rec} and the expression of $\pi_{m-1}$. Indeed,
\begin{align*}
\vphi(0)&=\sum_{i=1}^{m}\vphi(i)f(-i)=\sum_{i=1}^{m}f(-i)\sum_{j=0}^{i-1}\pi_j=\sum_{j=0}^{m-1}\pi_j\sum_{i=j+1}^{m}f(-i)=
\sum_{j=0}^{m-1}\pi_jF(-j-1)\\
&=(-1)^{m+1}\,\E(c\theta-X)\,\prod_{j=1}^{m-1}\frac{1}{\al_j-1}.
\end{align*}
The formula \eqref{eq:reccurence} is the rearranged version of \eqref{eq:main_rec}.

\end{proof}

\section{Numerical examples}\label{sec:examples}

In this section, we demonstrate some numerical outputs of statements formulated in the previous Section \ref{sec:results} when $X$ and $c\theta$ admit some chosen distributions. All the necessary calculations and visualizations are performed using the software \cite{Mathematica}.

\begin{ex}
Let $\mathbb{P}(X=0)=\mathbb{P}(X=1)=1/2$, $X\dist\theta$ and $c=2$. We give the exact values of $\vphi(0),\,\vphi(1),\,\vphi(2)$ and $\vphi(3)$ and provide an algorithm for arbitrary $\vphi(u)$ $u\in\mathbb{N}_0$ calculation.
\end{ex}

\bigskip

Let us observe the validity of the net profit condition $\mathbb{E}(2\theta-X)=1/2>0$. Then
$$
G_{X-2\theta}(s)=\left(\frac{1}{2}+\frac{s}{2}\right)\left(\frac{1}{2}+\frac{1}{2s^2}\right)=1 \quad \Rightarrow \quad s=1-\sqrt{2}=:\alpha
$$
and, according to Theorem \ref{thm:exact}, 
\begin{align*}
&\varphi(0)=\frac{\sqrt{2}}{4}\approx 0.354,\\
&\varphi(1)=2-\sqrt{2}\approx 0.586,\\
&\vphi(2)=2(\sqrt{2}-1)\approx 0.828,\\
&\vphi(3)=8-5\sqrt{2}\approx 0.929.
\end{align*}
We can proceed using the recurrence \eqref{eq:reccurence}
$$
\vphi(u)=4\left(\vphi(u-2)-\sum_{i=1}^{u-1}\vphi(i)f(u-2-i)\right),\,
u=2,\,3,\,\ldots,
$$
or employ the survival probability-generating function
$$
\Xi(s)=\frac{2-\sqrt{2}+\sqrt{2}s}{1+s-3s^2+s^3},\,s\neq \al,
$$
see Theorem \ref{thm:main_thm}, i.e.
$$
\vphi(u+1)=\frac{1}{u!}\lim_{s\to 0}\frac{d^u}{ds^u}\Xi(s),\,u\in\mathbb{N}_0.
$$

\begin{ex}
Let $p=1/2$,
$$
\mathbb{P}(X=k)=(1-p)^kp,\,k=0,\,1,\,\ldots
$$ 
and 
$$
\mathbb{P}(c\theta=k)=\binom{4}{k}p^k(1-p)^{4-k},\,k=0,\,1,\,\ldots,\,4,
$$
i.e. we assume the claim amount $X$ is geometric and the inter-occurrence time $\theta$ multiplied by premium $c$ is binomial distributed with the provided parameters. We give the exact values of $\vphi(0),\,\vphi(1),\,\vphi(2)$, $\vphi(3)$, $\vphi(4)$ and provide an algorithm for arbitrary $\vphi(u)$ $u\in\mathbb{N}_0$ calculation.
\end{ex}

\bigskip

Let us observe that the net profit condition is valid $\mathbb{E}(c\theta-X)=1>0$. Then, the equation
\begin{align*}
G_{X-c\theta}(s)=\frac{p}{1-(1-p)s}\left(1-p+\frac{p}{s}\right)^4=1,\,p=\frac{1}{2},
\end{align*}
has three roots inside the unit circle: 
\begin{align*}
\al_1 &= -0.15434 + 0.342115 i,\\ 
\al_2 &= -0.15434 - 0.342115 i,\\
\al_3 &=-0.289014.
\end{align*}
Note that the complex roots always appear in conjugate pairs, see \cite[Rem. 10]{GJS}.

According to Theorem \ref{thm:exact},
\begin{align*}
\vphi(0)&=-\prod_{j=1}^{3}\frac{1}{\al_j-1}=0.535194,\,
\vphi(1)=32\prod_{j=1}^{3}\frac{\al_j}{\al_j-1}=0.697233,\\
\vphi(2)&=32\left(-\frac{11}{64}\varphi(1)+\frac{\al_1\al_2\al_3-\al_1\al_2-\al_1\al_3-\al_2\al_3}{(\al_1-1)(\al_2-1)(\al_3-1)}\right)=0.802783,\\
\vphi(3)&=32\left(-\frac{11}{64}\vphi(2)-\frac{55}{128}\vphi(1)+\frac{\al_1\al_2\al_3-\al_1\al_2-\al_1\al_3-\al_2\al_3+\al_1+\al_2+\al_3}{(\al_1-1)(\al_2-1)(\al_3-1)}\right)\\
&=0.871536 
\end{align*}
We can proceed using the recurrence \eqref{eq:reccurence}
$$
\vphi(u)=32\left(\vphi(u-4)-\sum_{i=1}^{u-1}\vphi(i)f(u-4-i)\right),\,
u=4,\,5,\,\ldots,
$$
which gives 
$$
\varphi(4)=32\left(\vphi(0)-\vphi(1)\cdot\frac{65}{256}-\vphi(2)\cdot\frac{33}{128}-\vphi(3)\cdot\frac{9}{64}\right)=0.916321,
$$
or employ the survival probability-generating function (see Theorem \ref{thm:main_thm})
$$
\Xi(s)=\frac{0.0435771 + 0.224482 s + 0.516629 s^2 + 0.750506 s^3 - 0.535194 s^4}{0.0625 + 0.25 s + 0.375 s^2 + 0.25 s^3 - 1.9375 s^4 + s^5},\,s\neq \al_1,\,\al_2,\,\al_3,
$$
where $\pi_0,\,\pi_1,\,\pi_2,\,\pi_3$ can be obtained from $\vphi(0),\,\vphi(1),\,\vphi(2)$, $\vphi(3)$, $\vphi(4)$ or directly by Theorem \ref{thm:exact}.

\bigskip

\begin{ex}\label{ex:double_root}
Let $\mathbb{P}(c\theta=1)=p=1-\mathbb{P}(c\theta=3)$ and
$$
\mathbb{P}(X=0)=\frac{-1+p+\sqrt{1-p}}{2p}=1-\mathbb{P}(X=1),\,0<p\leqslant1.
$$
We find $\varphi(u),\,u\in\mathbb{N}_0$ and the survival probability generating function $\Xi(s)$, $|s|<1$.
\end{ex}

\bigskip

It is easy to see that the net profit condition holds for all $0< p <1$. The equation
$$
\left(\frac{-1+p+\sqrt{1-p}}{2p}+\frac{1+p-\sqrt{1+p}}{2p}s\right)
\left(\frac{p}{s}+\frac{1-p}{s^3}\right)=1
$$
has a double root
$$
s=-\frac{1-p}{1-p+\sqrt{1-p}}:=\al,\,0<p<1.
$$
inside the unit circle. We now setup the modified system \eqref{eq:system} by replacing its second line with derivative
\begin{align*}
\begin{pmatrix}
F(-3)+\al F(-2)+\al^2F(-1)&F(-3)\al+F(-2)\al^2&\al^2f(-3)\\
F(-2)+2\al F(-1)&F(-3)+2\al F(-2)&2\al f(-3)\\
f(-1)+2f(-2)+3f(-3)&f(-2)+2f(-3)&f(-3)
\end{pmatrix}
\begin{pmatrix}
\pi_0\\
\pi_1\\
\pi_2
\end{pmatrix}=
\begin{pmatrix}
&0\\
&0\\
&\mathbb{E}(c\theta-X)
\end{pmatrix},
\end{align*}
where the involved probability mass function and expectation are described by the assumed distribution of $X-c\theta$. The provided system has a solution $(\pi_0,\,\pi_1,\,\pi_2)=(1,\,0,\,0)$, which implies $\varphi(u)=1$ for all $u\in\mathbb{N}$ and $\Xi(s)=1/(1-s)$, $|s|<1$, while the recurrence \eqref{eq:reccurence} gives 
$$
\vphi(0)=\sum_{i=1}^{3}\vphi(i)f(-i)=f(-1)+f(-2)+f(-3)=1-f(0)=
\frac{1-p+\sqrt{1-p}}{2}
$$
and that is natural because under the assumed distributions the ruin can appear when $X=1$, $c\theta=1$ and $u=0$ only.

\bigskip

\begin{ex}\label{ex:puas}
Let $X\sim\mathcal{P}(1)$ and $c\theta\sim{P}(1.01)$ be Poisson distributed with parameters $\la=1$ and $\la=1.01$, i.e.
$$
\mathbb{P}(X=k)=\frac{e^{-1}}{k!},\,\mathbb{P}(c\theta=k)=e^{-1.01}\frac{1.01^k}{k!},\,k=0,\,1,\,\ldots
$$
We truncate $c\theta$ to have the finite support and calculate the exact survival probabilities $\varphi(0),\,\varphi(1),\,\ldots,\,\varphi(10)$. 
\end{ex}

\bigskip

As the proposed algorithm for the ultimate time survival probability calculation does not allow the inter-occurrence time of infinite support, for the illustrative purposes we define two new random variables $c\theta_{10}$ and $c\theta_{15}$:
\begin{align}\label{trunc_10}
&\mathbb{P}(c\theta_{10}=k)=\mathbb{P}(c\theta=k),\,k=0,\,1,\,\ldots,\,9,\,
\mathbb{P}(c\theta_{10}=10)=\mathbb{P}(c\theta\geqslant 10),\\
\label{trunc_15}
&\mathbb{P}(c\theta_{15}=k)=\mathbb{P}(c\theta=k),\,k=0,\,1,\,\ldots,\,14,\,
\mathbb{P}(c\theta_{15}=15)=\mathbb{P}(c\theta\geqslant 15).
\end{align}
It is easy to observe $\mathbb{E}(X-c\theta)=-0.01<0$. The net profit condition remains valid for the truncated distributions $c\theta_{10}$ and $c\theta_{15}$ too:
\begin{align*}
\mathbb{E}(X-c\theta_{10})&=-0.009999988,\\
\mathbb{E}(X-c\theta_{15})&=-0.00999999999998.
\end{align*}

Then, the equation
\begin{align}\label{eq:10}
e^{s-1}\left(\sum_{k=0}^{9}\frac{\mathbb{P}(c\theta=k)}{s^k}+\frac{\mathbb{P}(c\theta\geqslant 10)}{s^{10}}\right)=1
\end{align}
has nine roots inside the unit circle, the red points in Figure \ref{fig:roots} below, while  
\begin{align}\label{eq:15}
e^{s-1}\left(\sum_{k=0}^{14}\frac{\mathbb{P}(c\theta=k)}{s^k}+\frac{\mathbb{P}(c\theta\geqslant 15)}{s^{15}}\right)=1
\end{align}
has fourteen roots, the blue points in Figure \ref{fig:roots} below.
\begin{figure}[H]
\begin{center}
\includegraphics[scale=0.75]{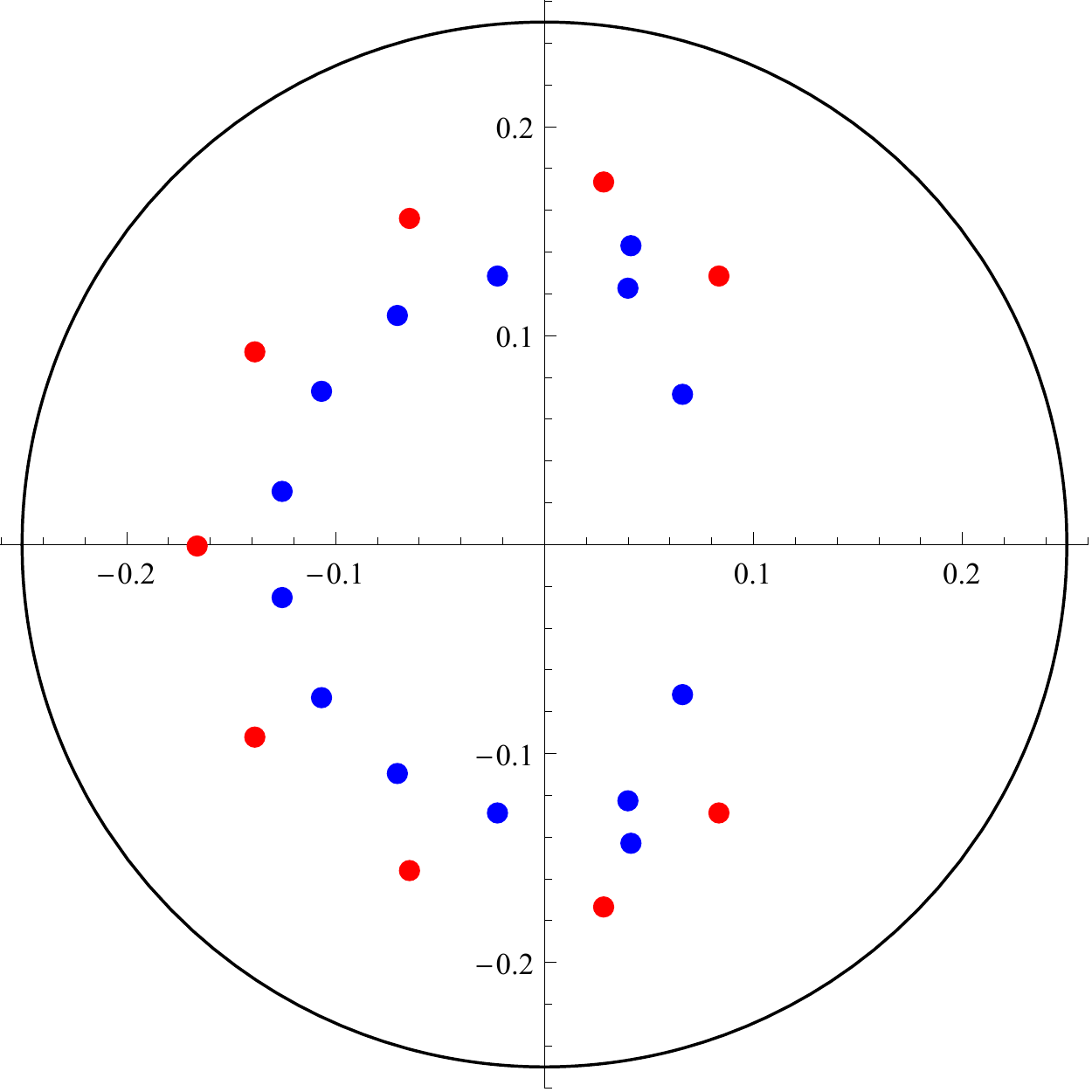}
\caption{Roots of \eqref{eq:10} (red) and \eqref{eq:15} (blue) inside the unit circle.}\label{fig:roots}
\end{center}
\end{figure}

Let us denote $\varphi_{10}(u)$ and $\varphi_{15}(u)$ the survival probabilities when the process \eqref{eq:process} is generated by $\{X,\,c\theta_{10}\}$ and $\{X,\,c\theta_{15}\}$ respectively. We then multiply the roots depicted in Figure \ref{fig:roots} as provided in Theorem \ref{thm:exact} and obtain the survival probabilities $\varphi_{10}(u)$ and $\varphi_{15}(u)$ when $u=0,\,1,\,\ldots,\,10$, see Table \ref{table1} below.

\begin{table}[h]
\centering
\begin{tabular}{|c|c|c|c|}
\hline
$u$&$\varphi_{10}(u)$&$\varphi_{15}(u)$&$\varphi_{15}(u)-\varphi_{10}(u)$\\
\hline
$0$&$0.0067795743$&$0.0067795818$&$7.5E-09$\\
\hline
$1$&$0.0145425921$&$0.0145456080$&$1.6E-08$\\
\hline
$2$&$0.0238700927$&$0.0238701187$&$2.6E-08$\\
\hline
$3$&$0.0334952018$&$0.0334952381$&$3.6E-08$\\
\hline
$4$&$0.0430669381$&$0.0430669845$&$4.6E-08$\\
\hline
$5$&$0.0525424876$&$0.0525425439$&$5.6E-08$\\
\hline
$6$&$0.0619232839$&$0.0619233499$&$6.6E-08$\\
\hline
$7$&$0.0712111444$&$0.0712112199$&$7.6E-08$\\
\hline
$8$&$0.0804070612$&$0.0804071458$&$8.5E-08$\\
\hline
$9$&$0.0895119320$&$0.0895120511$&$1.2E-07$\\
\hline
$10$&$0.0985266555$&$0.0985268429$&$1.9E-07$\\
\hline
\end{tabular}
\caption{The survival probabilities of $\varphi_{10}(u)$ and $\varphi_{15}(u)$.}\label{table1}
\end{table}
The obtained values of $\vphi_{10}(u)$ and $\vphi_{15}(u)$ in Table \ref{table1} indicate the negligible effect truncating the distribution of $c\theta$ up to $10$ in comparison to truncation up to $15$, see \eqref{trunc_10} and \eqref{trunc_15} and recall the estimates \eqref{ineq:est}.


\section{Acknowledgements} The author is grateful to Professor Jonas Šiaulys for carefully reading the first draft of the manuscript and raising many inaccuracies.

\bibliography{sn-bibliography}

\end{document}